\documentclass[10pt]{amsart}

\usepackage{amssymb,latexsym,amsmath}
\usepackage{graphics}                 
\usepackage{color}                    
\usepackage{hyperref}                 
\usepackage[all]{xy}

\numberwithin{equation}{section}
\theoremstyle{plain} 
\newtheorem{theorem}[equation]{Theorem}
\newtheorem{corollary}[equation]{Corollary}
\newtheorem{lemma}[equation]{Lemma}
\newtheorem{proposition}[equation]{Proposition}

\theoremstyle{definition}

\newtheorem{conjecture}[equation]{Conjecture}

\theoremstyle{remark}
\newtheorem{remark}[equation]{Remark}

\def\C{{\mathbb{C}}}

\def\P{{\mathbb{P}}}
\def\Q{{\mathbb{Q}}}

\def\Z{{\mathbb{Z}}}
\def\Ch{{\rm Ch}}

\def\supp{{\rm supp}}

\def\cZ{{\mathcal{Z}}}

\title{Holomorphic vector fields and Chow groups}
\author{Wenchuan Hu}
\date{August 1, 2019}


\address{
School of Mathematics\\
Sichuan University\\
Chengdu 610064\\
P. R. China
}

\email{huwenchuan@gmail.com}

\begin{document}

\begin{abstract}
We show that the chow group of $p$-cycles with rational coefficients are isomorphic to the corresponding rational homology groups for  smooth complex projective varieties carrying a holomorphic vector field with an isolated zero locus. As applications, we obtain  Chow groups and Lawson homology groups with rational coefficients and verify the Friedlander-Mazur conjecture and the Generalized Hodge conjecture for those varieties.
\end{abstract}

\maketitle
\pagestyle{myheadings}
 \markright{Holomorphic vector field and Chow groups}

\tableofcontents

\section{Introduction}


Let $V$ be a holomorphic vector field defined on a smooth projective algebraic variety
$X$. The zero subscheme  $Z$ is the subspace of $X$ defined by the ideal generated by $V\mathcal{O}_X$ and we denote it by $X^{V}$.
The existence of a holomorphic vector field with zeroes on a
compact complex manifold imposes restrictions on the topology  of the manifold. For example, the Hodge numbers
$h^{p,q}(X)=0$ if $|p-q|>\dim Z$ (see \cite{Carrell-Lieberman} and references therein).

According to Lieberman (\cite{Lieberman1}), a holomorphic vector field $V$ on a complex algebraic projective variety $X$ with nonempty zeroes is equivalent to
the 1-parameter group $G$ generated by $V$ is a product of $\C^*$'s and at most one $\C$'s. This induces us to study the structure of $X$ admitting a
$\C^*$-action or a $\C$-action.

For a smooth complex projective variety $X$ admitting a $\C^*$-action, Bialynicki-Birula structure theorem describes
the relation between the structure of $X$ and that of the fixed points set(\cite{Bialynicki-Birula2}). In particular, there is a homology (resp. Chow groups, Lawson homology, etc.) basis formula
connecting the homology groups (resp. Chow groups, Lawson homology groups, etc.) of $X$ and those of $X^{\C^*}$.

It is expected that a similar decomposition holds for a smooth complex projective variety $X$ admits a $\C$-action.
However, very little is known in this direction.

In section \ref{Section3} we will use the method on the  decomposition of  diagonal by Bloch and Srinivas(\cite{Bloch-Srinivas}, Bloch credits the idea to Colliot-Th\'{e}l\`{e}ne.) to compute the Chow groups of $X$ in the case that $X^{V}$ are isolated points. In rational coefficients, the Chow groups of $X$ are identified with their corresponding homological groups with rational coefficients.
\begin{theorem}\label{Thm1}
Let $X$ be a smooth complex projective algebraic variety which admits a holomorphic
vector field $V$ whose zero set $Z$ is isolated and nonempty. Then the cycle class map
$\Ch_p(X)\otimes \Q\to H_{2p}(X,\Q)$ is injective for all $p$.
\end{theorem}

Furthermore, the injectivity of $\Ch_p(X)\otimes \Q\to H_{2p}(X,\Q)$ for all $p$ implies that they are isomorphisms for all $p$.
As applications, we show that the corresponding Chow groups, Lawson homology groups and singular homology groups are isomorphic in rational coefficients.
Moreover,  the Friedlander-Mazur conjecture and the generalized Hodge conjecture hold for such an $X$.

If $X$ is singular, much weaker results can be obtained in general when $X$ either admits a $\C^*$-action or a $\C$-action. The details and application in this direction will appear in a subsequent paper.

\section{Invariants under the additive group action}
\label{section2}
Let $X$ be a possible singular complex
projective algebraic variety $X$ admitting an additive group action.
Our main purpose is to compare certain algebraic and topological invariants (such as the Chow group of zero cycles, Lawson homology, singular homology, etc.) of $X$ to those of the fixed point set $X^{\C}$.
If $X$ is smooth projective, most of topological invariants are studied and  computed in details, but some of algebraic invariants are still hard to identified.
 Some of those invariants have been investigated even if $X$ is singular. In this section, we will identify some of these invariants including the Chow groups of zero cycles, Lawson homology
 for 1-cycles, singular homology with integer coefficients, etc.

\subsection{A-equivalence}

Let $A$ be a fixed  complex quasi-projective algebraic variety. Recall that  an algebraic scheme $X_1$ is \textbf{simply
$A$-equivalent} to an algebraic variety $X_2$ if $X_1$ is isomorphic to a closed subvariety $X_2'$ of
$X_2$ and there exists an isomorphism $f: X_2- X_2'\to Y \times A$, where $Y$ is an algebraic variety.
The smallest equivalence relation containing the relation of simple $A$-equivalence is called
the \textbf{A-equivalence} and we denote it by $\sim$ (see \cite{Bialynicki-Birula1}). A result of Bialynicki-Birula
says that $X{\sim} X^{\C}$ if $X$ is a quasi-projective variety admitting a $\C$-action.
A similar statement  holds for  $X$ admitting $\C^*$-action.
From this, Bialynicki-Birula showed that  $H^0(X,\Z)\cong H^0(X^{\C},\Z)$ and $H^1(X,\Z)\cong H^1(X^{\C},\Z)$ in the case that $X$ admits a $\C$-action, where
$\chi(X)=\chi(X^{\C^*})$ in the case that $X$ admits a $\C$-action (see \cite{Bialynicki-Birula1}). Along this route, more additive invariants has been
calculated for varieties admits a $\C$ or $\C^*$-action (see \cite{Hu}).
\subsection{Chow Groups}

Let $X$ be any complex projective variety or scheme of dimension $n$ and let $\cZ_p(X)$  be the group of algebraic $p$-cycles on $X$.
 Let $\Ch_p(X)$ be the Chow group of $p$-cycles on $X$, i.e.
$\Ch_p(X)=\cZ_p(X)/{\rm \{rational ~equivalence\}}$.  Set $\Ch_p(X)_{\Q}:=\Ch_p(X)\otimes \Q$, $\Ch_p(X)=\bigoplus_{p\geq0} \Ch_p(X)$.
Let $A_p(X)$ be the space of $p$-cycles on $X$ modulo
the algebraic equivalence, i.e. $A_p(X)=\cZ_p(X)/{\sim_{alg}}$, where $\sim_{alg}$ denotes the algebraic equivalence.
For convenience, set $A_p(X)_{\Q}:=A_p(X)\otimes \Q$,    $\Ch^q(X)_{\Q}=\Ch_{\dim X-q}(X)_{\Q}$ and
$A^q(X)_{\Q}=A_{\dim X-q}(X)_{\Q}$. Let $cl_p: \Ch_p(X)\to H_{2p}(X,\Z)$ be the cycle class map. Denoted by $\Ch_p(X)_{hom}:=\ker(cl_p)$.
There are all kinds of  functorial properties including pull-forward for morphisms, pull-back for flat morphisms, homotopy invariance property
and a well-defined intersection theory on smooth projective varieties, etc.
For more details on Chow theory, the reader is  referred to Fulton (\cite{Fulton}).

Recall that (cf. \cite{Bloch}) for each $m\geq 0$, let
$$
\Delta[d]:=\{t\in \C^{d+1}|\sum_{i=0}^{m} t_i=1\}\cong \C^d.
$$
and let $z^l(X,d)$ denote the free abelian group generated by irreducible subvarieties of codimension-$l$ on $X\times \Delta[d]$ which
meets $X\times F$ in proper dimension for each face $F$ of $\Delta[d]$. Using intersection and pull-back of algebraic cycles,
we can define face and degeneracy relations and obtain a simplicial abelian group structure for $z^l(X,d)$.
Let $|z^l(X,*)|$ be the geometric realization of $z^l(X,*)$.  Then the higher Chow group is defined as
\begin{equation*}
\Ch^l(X,k):=\pi_k(|z^l(X,*)|)
\end{equation*}
 and set $ \Ch_{l}(X,k):=\Ch^{n-l}(X,k).$  In particular, $ \Ch_{l}(X,0)=\Ch_l(X)$.

Note that  there exists a long exact sequence for higher Chow groups (see \cite{Bloch}):
\begin{equation}\label{equ2.1}
...\to\Ch_r(Y,k)\to\Ch_r(X,k)\to\Ch_r(U,k)\to\Ch_r(Y,k-1)\to...
\end{equation}
for any triple $(X,Y, U)$, where $X$ is a quasi-projective variety, $Y\subset X$ a closed subvariety and $U\cong X-Y$.

Moreover, there are homotopy invariance for higher Chow groups, i.e.,
\begin{equation}\label{equ2.2}
\Ch_r(X,k)\cong \Ch_{r+1}(X\times \C,k)
\end{equation}
for any quasi-projective variety $X$.

\begin{proposition}\label{Prop2.2}
Let $X$ be a (possible singular) connected  complex projective variety. If $X$ admits a $\C$-action with isolated fixed points, then $\Ch_0(X)\cong \Z$.
\end{proposition}
\begin{proof}
Since  $X$  admits a $\C$-action with isolated fixed points, there exists a $\C$-invariant Zariski open set $U\subset X$  such that $U\cong U'\times \C$ (see \cite{Bialynicki-Birula1}).
Such  $U$ and $U'$  can be assumed to be non-singular if necessary.
The compliment $Z:=X-U$ is also $\C$-invariant. We have the long exact sequence for homology groups
$$
        ...\to H_1^{BM}(U,\Z)\to H_0(Z,\Z)\to H_0(X,\Z)\to H_0^{BM}(U,\Z)\to 0.
$$
Note that $H_0^{BM}(U,\Z)=H_0^{BM}(U'\times\C,\Z)\stackrel{P.D.}{\cong} H_{2n}(U'\times\C,\Z)=0$ and
$H_1^{BM}(U,\Z)=H_1^{BM}(U'\times\C,\Z)\stackrel{P.D.}{\cong} H_{2n-1}(U'\times\C,\Z)=H_{2n-1}(U',\Z)=0$, where P.D. denotes the Poincare Duality for Borel-Moore homology and the last
equality holds since $\dim_{\C} U'=n-1$. Hence $H_0(Z,\Z)\cong H_0(X,\Z)\cong \Z$ and so $Z$ is connected in the complex topology.

By the localization sequence of Chow groups, we have the exact sequence
$$\Ch_0(Z)\to \Ch_0(X)\to \Ch_0(U)\to 0.$$

By the induction hypothesis,
one has $\Ch_0(Z)\cong \Z$.
  Since $U\cong U'\times \C$, we have $\Ch_0(U)\cong\Ch_0(U'\times \C)=0$. Therefore, $\Z\to \Ch_0(X)$ is surjective. So $\Ch_0(X)$ is isomorphic to
 $\Z$ or $\Z_m$ for some positive integer $m$.

By the following commuative
diagram
$$
\xymatrix{ \Ch_0(Z)\ar[r]\ar[d]^{\cong} &\Ch_0(X)\ar[r]\ar[d] &\Ch_0(U)\ar[d]^{=}\ar[r] &0\\
 H_0(Z,\Z)\ar[r] &H_0(X,\Z)\ar[r] &H_0^{BM}(U,\Z)\ar[r] &0,
}
$$
we have $\Ch_0(X)\cong\Z$ since $\Ch_0(U)\cong H_0^{BM}(U,\Z)=0$ and $H_0(X,\Z)\cong \Z$.
This completes the proof of the proposition.
\end{proof}

\begin{remark}
There is another method to show that $\Ch_0(X)\cong\Z$. Since  $X$ is connected and  admits a $\C$-action with isolated fixed points, the fixed points $X^{\C}$ must be exactly one point $p_0$(\cite[Cor. 1]{Bialynicki-Birula1}).
The closure of the orbit of any a point $y\in X$ contains the fixed point $p_0$. Therefore, $y$ is rationally connected to $p_0$ and hence by definition of the Chow group, one has $\Ch_0(X)\cong\Z$.
In particular, if $X$ is a smooth projective variety, then $X$ is rationally connected(see \cite[Prop. 3]{Hwang1}) and so $\Ch_0(X)\cong\Z$.
\end{remark}

\begin{remark}
More generally, by using the same method, we can show that if $X$ admits a $\C$-action with  fixed points $X^{\C}$, then $\Ch_0(X)\cong \Ch_0(X^{\C})$.
\end{remark}

\subsection{Lawson homology}

The \emph{Lawson homology}
$L_pH_k(X)$ of $p$-cycles for a projective variety is defined by
$$L_pH_k(X) := \pi_{k-2p}({\mathcal Z}_p(X)) \quad for\quad k\geq 2p\geq 0,$$
where ${\mathcal Z}_p(X)$ is provided with a natural topology (cf.
\cite{Friedlander1}, \cite{Lawson1}). It has been extended to define for a quasi-projective  variety by Lima-Filho (cf. \cite{Lima-Filho}).
For general background, the reader is referred to Lawson' survey paper \cite{Lawson2}.
The definition of Lawson homology has been extended to negative integer $p$. Formally for $p<0$, we have
$L_pH_k(X)=\pi_{k-2p}({\mathcal Z}_{0}(X\times \mathbb{C}^{-p}))=H^{BM}_{k-2p}(X\times\mathbb{C}^{-p},\Z)=H^{BM}_k(X,\Z)=L_0H_k(X)$
(cf. \cite{Friedlander-Haesemesyer-Walker}), where $H^{BM}_{*}(-,\Z)$ denotes the Borel-Moore homology with $\Z$-coefficients.

In \cite{Friedlander-Mazur}, Friedlander and Mazur showed that there
are  natural transformations, called \emph{Friedlander-Mazur cycle class maps}
\begin{equation}\label{eq01}
\Phi_{p,k}:L_pH_{k}(X)\rightarrow H_{k}(X,\Z)
\end{equation}
for all $k\geq 2p\geq 0$.

Recall that Friedlander and Mazur constructed a map called the $s$-map $s:L_pH_k(X)\to L_{p-1}H_k(X)$ such that the cycle class map
$\Phi_{p,k}=s^p$ (\cite{Friedlander-Mazur}).
Explicitly, if $\alpha\in L_pH_k(X)$ is represented by the homotopy class of a continuous map  $f:S^{k-2p}\to \cZ_p(X)$, then
$\Phi_{p,k}(\alpha)=[f\wedge S^{2p}]$, where $S^{2p}=S^2\wedge\cdots\wedge S^2$ denotes the $2p$-dimensional topological sphere.

Set
{$$
\begin{array}{llcl}
&L_pH_{k}(X)_{hom}&:=&{\rm ker}\{\Phi_{p,k}:L_pH_{k}(X)\rightarrow
H_{k}(X)\};\\
&L_pH_{k}(X)_{\Q}&:=&L_pH_{k}(X)\otimes\Q.
\end{array}
 $$}

Denoted by $ \Phi_{p,k,\Q}$ the map $ \Phi_{p,k}\otimes{\Q}:L_pH_{k}(X)_{\Q}\rightarrow H_{k}(X,\Q) $.
The \emph{Griffiths group} of dimension $p$-cycles is defined to be
$$
{\rm Griff}_p(X):={\mathcal Z}_p(X)_{hom}/{\mathcal Z}_p(X)_{alg}.$$

Set
$$
\begin{array}{lcl}
{\rm Griff}_p(X)_{\Q}&:=&{\rm Griff}_p(X)\otimes\Q;\\
{\rm Griff}^q(X)&:=&{\rm Griff}_{n-q}(X);\\
{\rm Griff}^q(X)_{\Q}&:=&{\rm Griff}_{n-q}(X)_{\Q}.
\end{array}
$$

It was proved by Friedlander \cite{Friedlander1} that, for any
smooth projective variety $X$, $$L_pH_{2p}(X)\cong {\mathcal
Z}_p(X)/{\mathcal Z}_p(X)_{alg}=A_p(X).$$

Therefore
\begin{eqnarray*}
L_pH_{2p}(X)_{hom}\cong {\rm Griff}_p(X).
\end{eqnarray*}

\begin{proposition}\label{Prop2.4}
Under the same assumption as Proposition \ref{Prop2.2},  we have
$$L_1H_k(X)\cong H_k(X,\Z)$$ for all $k\geq 2$. In particular, ${\rm Griff}_1(X)=0$.
\end{proposition}
\begin{proof}
Since the natural transform $\Phi_{p,k}:L_pH_k(-)\to  H_k(-,\Z)$ is a natural transform and  there exists a long localization sequence of Lawson homology, one has the
following commutative diagram of long exact sequences (see \cite{Lawson1}, \cite{Lima-Filho})
{\tiny
$$
\xymatrix{L_1H_{k+1}(U)\ar[r]\ar[d]^{\Phi_{p,k+1}}& L_1H_k(Z)\ar[r]\ar[d]^{\cong} &L_1H_k(X)\ar[r]\ar[d]^{\Phi_{p,k}} &L_1H_k(U)\ar[d]^{\Phi_{p,k}}\ar[r] &L_1H_{k-1}(Z)\ar[d]^{\cong}\\
H_{k+1}^{BM}(U,\Z) \ar[r]& H_{k}(Z,\Z)\ar[r] &H_{k}(X,\Z)\ar[r] &H_{k}^{BM}(U,\Z)\ar[r] & H_{k-1}(Z),
}
$$}
\noindent
where $Z$ and $U$ are the same as in the proof of Proposition \ref{Prop2.2}. Since $Z$ admits $\C$-action with isolated point and connected as shown in Proposition \ref{Prop2.2},
 we have by  the induction hypothesis that $\Phi_{1,k}:L_1H_k(Z)\cong H_k(X,\Z)$ for all $k\geq 2$.

Note that $L_1H_k(U)=L_1H_k(U'\times\C)\stackrel{S}{\cong} L_0H_{k-2}(U')\stackrel{D.T.}{\cong} H_{k-2}^{BM}(U')$ and
$L_1H_{k+1}(U)=L_1H_{k+1}(U'\times\C)\stackrel{S}{\cong} L_0H_{k-1}(U')=H_{k-1}^{BM}(U',\Z)$, where $S$ denotes the Suspension isomorphism  for Lawson homology and D.T denotes the
Dold-Thom theorem.  Hence $L_1H_k(X)\cong H_k(X,\Z)$ follows from the Five Lemma. From the fact that ${\rm Griff}_1(X)\cong \ker \{\Phi_{1,2}:L_1H_k(X)\cong H_k(X,\Z)\}$, we get ${\rm Griff}_1(X)=0$.
This completes the proof of the proposition.
\end{proof}

\begin{remark}
The isomorphism $L_0H_k(X)\cong H_k(X,\Z)$ holds for any integer $k\geq 0$, which is the special case of the Dold-Thom Theorem.
\end{remark}

\begin{remark}
The assumption of ``connectedness" in Proposition \ref{Prop2.4} is not necessary. By the same reason, we can remove the connectedness in Proposition \ref{Prop2.2}, while the conclusion
``$\Ch_0(X)\cong \Z$" would
be replaced by ··$\Ch_0(X)\cong H_0(X,\Z)$”.
\end{remark}

\begin{remark}
In case that $X$ is smooth projective, the statement ``$\Ch_0(X)\cong \Z$" implies that $L_1H_k(X)_{hom}\otimes \Q=0$ for all $k\geq 2$ (see \cite{Peters}).
However, if $X$ is singular projective variety, ``$\Ch_0(X)\cong \Z$" and  ``$L_1H_k(X)\cong H_k(X,\Z)$" are independent statements
in the sense that we can find examples such that one  holds but the other fail (see \cite{Hu2}).
\end{remark}

From the proof of Proposition, we have actually shown the following result.
\begin{corollary}\label{coro2.11}
Let $X$ be a (possible singular,  reducible)  complex projective variety admitting  a $\C$-action. Let $X^{\C}$ be the set of  fixed points.
If we have $L_1H_k(X^{\C})\cong H_k(X^{\C},\Z)$  for all $k\geq 2$, then
$$L_1H_k(X)\cong H_k(X,\Z)$$ for all $k\geq 2$.
\end{corollary}
\begin{proof}
In the proof Proposition \ref{Prop2.4}, the hypothesis induction we actually used is the isomorphism ``$L_1H_k(Z)\cong H_k(Z,\Z)$" if $\dim Z<\dim X$ and $Z^{\C}=X^{\C}$,
which is implied by the assumption $X^{\C}$ is isolated. Now the isomorphism ``$L_1H_k(Z)\cong H_k(Z,\Z)$" if $\dim Z<\dim X$ and $Z^{\C}=X^{\C}$ is our assumption. This
completes the proof of the corollary.
\end{proof}

\subsection{The virtual Betti  and Hodge numbers }
 Recall that
 the \emph{virtual Hodge polynomial} $H:Var_{\C}\to \Z[u,v]$ is defined by  the following properties:
\begin{enumerate}
 \item $H_X(u,v):=\sum_{p,q}(-1)^{p+q}\dim H^{q}(X,\Omega_X^p)u^pv^q$ if $X$ is nonsingular and  projective (or complete).
\item  $H_X(u,v)=H_U(u,v)+H_Y(u,v)$ if $Y$ is a closed algebraic subset of $X$ and $U=X-Y$.
\item  If $X=Y\times Z$, then $H_X(u,v)=H_Y(u,v)\cdot H_Z(u,v)$.
\end{enumerate}

For example, $H_{\P^1}(u,v)=1+uv$, $H_{\P^0}(u,v)=1$, $H_{\C}(u,v)=(1+uv)-1=uv$, and $H_{\C^m}(u,v)=(uv)^m$.
For a smooth algebraic curve $C$ of genus $g$,  $H_{C}(u,v)=uv+gu+gv+1$.
The existence and uniqueness of such a polynomial follow from Deligne's Mixed Hodge theory (see \cite{Deligne1,Deligne2}).
The coefficient of $u^pv^q$ of $H_X(u,v)$ is called the \emph{virtual Hodge $(p,q)$-number} of $X$ and we denote it by $\tilde{h}^{p,q}(X)$.
Note that from the definition, $\tilde{h}^{p,q}(X)$ coincides with the usual Hodge number $(p,q)$-number ${h}^{p,q}(X)$ if $X$ is a smooth
projective variety. The sum $\tilde{\beta}^k(X):=\sum_{i+j=k}\tilde{h}^{p,q}(X)$ is called the $k$-th \emph{virtual Betti number} of $X$.
The \emph{ virtual Poincar\'{e} polynomial} of $X$ is defined to be
$$\widetilde{P}_X(t):=\sum_{k=0}^{2\dim_{\C} X} \beta^k(X)t^k,$$
which coincides to the usual Poincar\'{e} polynomial defined through the corresponding usual Betti numbers.

 Let $X$ be a (possible singular) connected  complex projective variety. It was shown in \cite{Bialynicki-Birula1} that if $X$ admits a $\C$-action with isolated fixed points,
 then $H^1(X,\Z)=0$. Furthermore, it was shown in \cite{Hu} that the virtual Hodge numbers $\tilde{h}^{p,0}(X)$, $\tilde{h}^{0,q}(X)$ vanish for all $q\neq 0$.
This implies that the there is no holomorphic form $p$-forms on $X$ if $X$ is smooth projective and has a holomorphic vector field $V$ such that $p>\dim zero(V)$(cf. \cite{Howard}).

Moreover, the virtual Betti nubmer $\tilde{\beta}^1(X)=0$.
   The following proposition says that $H_1(X,\Z)=0$,
which may be implied  in  literatures but to our knowledge it has  not been mentioned
explicitly elsewhere. By  the universal coefficient theorem,  it is equivalent to $H^1(X,\Z)=0$ and the torsion of  $H_1(X,\Z)$ is zero.

\begin{proposition}\label{Prop2.12}
Under the same assumption as Proposition \ref{Prop2.2},  we have $$ H_1(X,\Z) = 0.$$
\end{proposition}
\begin{proof}
Using the notations as in Proposition \ref{Prop2.2}, we have the following long exact sequence of Borel-Moore homology
$$
        ...\to H_1^{BM}(U,\Z)\to H_0(Z,\Z)\to H_0(X,\Z)\to H_0^{BM}(U,\Z)\to 0.
$$
Since $H_0^{BM}(U,\Z)=0$ and $H_1^{BM}(U,\Z)=0$ (see the proof of Proposition \ref{Prop2.2}).
Hence the sequence reduces to the following long exact sequence
\begin{equation}\label{eqn2.10}
 ... \to H_2^{BM}(U,\Z)\to H_1(Z,\Z)\to H_1(X,\Z)\to 0.
\end{equation}

Note that
$$
\begin{array}{ccl}
H_2^{BM}(U,\Z)&=&H_2^{BM}(U'\times\C,\Z)\\
&\stackrel{P.D.}{\cong} &H_{2n-2}(U'\times\C,\Z)\\
&=&H_{2n-2}(U',\Z)\\
&=&\Z^m,
\end{array}
$$
where $m$ is the number of connected components of $U'$.

Since $H_1(Z,\Z)=F_1(Z)\oplus T_1(Z)$ and $H_1(X,\Z)=F_1(X)\oplus T_1(X)$, where $F_1(Z)$ (resp.  $T_1(Z)$) denotes
the free (resp. torsion) part of $H_1(Z,\Z)$.
From Bialynicki-Birula's result, one has  $F_1(Z)\cong F_1(X)=0$. These together with Equation \eqref{eqn2.10} yield
\begin{equation}\label{eqn2.11}
 ...\to \Z^m\to T_1(Z,\Z)\to T_1(X,\Z)\to 0.
\end{equation}
Hence $ T_1(Z,\Z)\to T_1(X,\Z)$ is surjective. By the induction hypothesis, one has $T_1(Z,\Z)=0$ and so $H_1(X,\Z)=T_1(X,\Z)=0$.
This completes the proof of the proposition.
\end{proof}

From the proof of Proposition \ref{Prop2.12},  we actually show that  if  $X$ is a (possible singular) connected  complex quasi-projective variety admitting
 a $\C$-action with isolated fixed points or even no fixed points, then $H_1^{BM}(X,\Z) =0$. More generally, we actually get the following result.
\begin{corollary}
Let $X$ be a complex quasi-projective variety with a $\C$-action, then we have $H_1^{BM}(X,\Z) \cong H_1^{BM}(X^{\C},\Z)$.
\end{corollary}

\section{Chow groups of a holomorphic vector field with isolated zeroes}
\label{Section3}
In this section, $X$ is a nonsingular complex projective variety unless otherwise specified.
A holomorphic vector field $V$ on $X$ means that $V\in H^0(X,\mathcal{T}_X)$, where $\mathcal{T}_X$ is the
tangent sheaf of $X$.

If $X$ admits a holomorphic vector field with isolate zeroes, then  the 1-parameter group $G$ by the vector field
is   $(\C^*)^k\times (\C)^r$ for some $k\geq 0$ and $0\leq r\leq 1$ (see \cite{Lieberman1}).
If $r=0$, then $X$ admits a torus action with isolated fixed points and hence $X$ admits a cellular decomposition (see \cite{Bialynicki-Birula1})
and so $\Ch_p(X)\cong H_{2p}(X,\Z)$. If $r=1$, then we write $G_1=(\C^*)^{k-1}\times \C$ and so $G=G_1\times \C^*$. Then $X_1:=X^{\C^*}$ is
nonsingular and $Z:=(X_1)^{G_1}$ is the isolated zero set. Therefore,  Theorem \ref{Thm1} reduces to the following theorem.

\begin{theorem}\label{Th3.1}
Let $X$ be a smooth connected complex projective algebraic variety which admits a  $\C$-action  whose fixed point set $Z$ is isolated and nonempty. Then the cycle class map
$\Ch_p(X)\otimes \Q\to H_{2p}(X,\Q)$ is injective for all $p$.
\end{theorem}

\subsection{The decomposition of Diagonal }
 According to Bloch and Srinivas,  the triviality of the Chow group of a projective variety $X$ gives rise to the
 decomposition of diagonal in $X\times X$.

\begin{proposition}[\cite{Bloch-Srinivas}]\label{Prop3.2}
 Let $ \Delta_X\subset X \times X$ be the diagonal. If $\Ch_0(X)\cong\Z$,  then there exists
an integer $N > 0$, a divisor $D \subset X$, and $n$-cycles $\Gamma_1$,  $\Gamma_2$ on $X \times X$ such that
$\supp(\Gamma_0) \subset X\times p_0$, $\supp (\Gamma^1) \subset D \times X$ and
$$\Delta_X = \Gamma_0 + \Gamma^1$$
in $\Ch^n(X \times X)\otimes \Q$.
\end{proposition}
There many variants and applications of this technique, including the  generalizations given by Paranjape \cite{Paranjape},  Laterveer \cite{Laterveer}, etc.
\begin{proposition}\label{Prop3.3}
Let $X$ be a smooth projective variety.
Assume that for $p\leq s$, the maps
$$cl: \Ch_p(X)\otimes {\mathbb{Q}}\rightarrow H^{2n-2p}(X,{\mathbb{Q}})$$
are injective. Then there exists a decomposition
$$\Delta_X= \Gamma_{0}+\cdots+\Gamma_{s}+\Gamma^{s+1}\in {\rm CH}^n(X\times X)\otimes {\mathbb{Q}},$$
where $\Gamma_{p}$ is supported in $V_{n-p}\times W_{p}$,
$p=0,\cdots, s$ with $\dim V_{n-p}=n-p$ and $\dim W_{p}=p$, and
$\beta$ is supported in $V_{n-s-1}\times X$.
\end{proposition}

\begin{proposition}[cf. \cite{Laterveer}, \cite{Peters}]\label{Prop3.4}\label{Prop3..4}
Assume that $X$
and $Y$ are smooth projective varieties and let $\alpha\subset
X\times Y$ be  an irreducible cycle of dimension $\dim(X)=n$,
supported on $V\times W$, where, $V\subset X$ is a subvariety of
dimension $v$ and $W\subset Y$ a subvariety of dimension $w$. Let
$\tilde{V}$ , resp. $\tilde{W}$ be a resolution of singularities of
$V$, resp. $W$ and let $\tilde{i}:\tilde{V}\rightarrow X$ and
$\tilde{j}:\tilde{W}\rightarrow Y$ be the corresponding morphisms. With
$\tilde{\alpha}\subset \tilde{V}\times \tilde{W}$ the proper
transform of $\alpha$ and $p_1$, resp. $p_2$ the projections from
$X\times Y$ to the first. resp. the second factor, there is a
commutative diagram

$$
\begin{array}{cccc}
  \Ch_{p-n+v+w}(\tilde{V}\times \tilde{W}) &
   \stackrel{\tilde{\alpha}_*}{\longrightarrow} & \Ch_p(\tilde{V}\times \tilde{W})   \\

 \uparrow p_1^*&      & \downarrow  (p_2)_*    \\

 \Ch_{p-n+v}(\tilde{V})&   & \Ch_p(\tilde{W})   \\

 \uparrow \tilde{i}^*&   & \downarrow \tilde{j}_*   \\

  \Ch_p(X) & \stackrel{{\alpha}_*}{\longrightarrow} & \Ch_p(Y).

\end{array}
$$
Here $\tilde{i}^*$ is induced by the Gysin homomorphism, $p_1^*$ is the
flat pull-back, and $(p_2)_*$  and $\tilde{j}_*$ come from proper push
forward. In particular, $\alpha_*=0$ if $p<n-v$ or if $p>w$.
Moreover, $\alpha_{n-v}$ acts trivially on $\Ch_{n-v}(X)_{hom}$ if $\Ch_0(\widetilde{V})_{hom}=0$,
while $\alpha_w$ acts trivially on $\Ch_{w}(X)_{hom}$.

 \end{proposition}

\begin{proof}
For the proof of the commutative diagram for Chow groups,  it was shown in \cite[Thm. 1.7]{Laterveer}.
The statement in this proposition  is the analogue for Lawson homology given in the Proposition 12 in \cite{Peters}.

The one but last assertion follows from the fact that $\Ch_{m-n+v}(\widetilde{V}) = 0$
for $m-n+v < 0$ and $\Ch_m(\widetilde{W}) = 0$ if $m > w$. The final assertion follows
from the fact that for all varieties $Z$, one has $\Ch_t (Z)_{hom} = 0 $ for $t = \dim Z$,
while also $\Ch_0 (\widetilde{V})_{hom} = 0$ by assumption.
\end{proof}

\begin{lemma} \label{lemma3.5}
Under the same assumption in Proposition \ref{Prop3.4},  we obtain that  $\alpha_{n-v}$ acts trivially on $\Ch_{n-v}(X)_{hom}$ if $\Ch_0( {V})_{hom}=0$.
\end{lemma}
\begin{proof}
First of all, it is enough to show that for the morphism $\tilde{i}=i\circ\sigma:\widetilde{V}\stackrel{\sigma}{\to} V\stackrel{i}{\hookrightarrow} X$, all the
pullback maps $\sigma^*:\Ch^k(V)\to \Ch^k(\widetilde{V})$,  $i^*:\Ch^k(X)\to \Ch^k(V)$ and $\tilde{i}^*=(i\circ\sigma)^*:\Ch^k(X)\to \Ch^k(\widetilde{V})$ are defined
and hence $(i\circ\sigma)^*=\sigma^*\circ i^*$. Once these are proved, then the map $\Ch^k(X)\to \Ch^k(\widetilde{V})$ factors through $\Ch^k(V)$ and one
gets the triviality of the action  $\alpha_{n-v}$ on $\Ch_{n-v}(X)_{hom}$ under the assumption $\Ch_0( {V})_{hom}=0$. To see that $i^*$ is well-defined,
we note that $X$ is smooth, so the graph of $i$ is $\Gamma_i\stackrel{j}{\hookrightarrow} V\times X$, which  is a locally complete intersection in $V\times X$ and the projection $pr_X:V\times X\to X$ is flat. Hence
$i^*$ is defined to be the composition $\Ch^k(X)\stackrel{pr_X^*}{\to}\Ch^k(V\times X)\stackrel{j^*}\to \Ch^k(\Gamma_i) =\Ch(V)$ (see \cite[p.258]{Voisin}).
To see that $\sigma^*$ is well-defined, we consider the factorization of $\sigma$ as the composition $\widetilde{V}\to \Gamma_{\sigma}\stackrel{i_{\sigma}}{\hookrightarrow} \widetilde{V}\times V\stackrel{pr_V}{\to} V$,
where the first map is the isomorphism  of $V$ to the graph of $\sigma$, $i_{\sigma}$ is the inclusion of the graph into the product $\widetilde{V}\times V$ and $pr_V$ is the projection to $V$.
Since $\widetilde{V}$ is smooth, $pr_V$ is a smooth map. Hence $i_{\sigma}$ is a regular embedding and we have a well-defined pullback map on Chow groups $\Ch^k(\widetilde{V}\times V)\to \Ch^k(\Gamma_{\sigma})$ (see \cite[B.7.6]{Fulton}). Now
$\sigma^*:\Ch^k(V)\to \Ch^k(\widetilde{V})$ is the composition $ \Ch^k(V)\stackrel{pr_V^*}{\to} \Ch^k(\widetilde{V}\times V )\stackrel{i_{\sigma}^*}{\to} \Ch^k(\Gamma_{\sigma})\to \Ch^k(\widetilde{V})$.
Since both $\sigma^*$ and $ i^*$ are well-defined, the equality  $(i\circ\sigma)^*=\sigma^*\circ i^*$ follows directly from their definitions.
Therefore, we have  the following commutative diagram

$$
\xymatrix{ \Ch^k(X) \ar[r]^{i^*} \ar@{=}[d] & \Ch^k(V) \ar[d]^{\sigma^*} \\
    \Ch^k(X)\ar[r]^{(i\circ \sigma)^*}& \Ch^k(\widetilde{V}),
}
$$
where $k=\dim (V)$.
\end{proof}


Now we can give a proof of  Theorem \ref{Th3.1}.
\begin{proof}[Proof of Theorem \ref{Th3.1}]
By hypothesis, the fixed point set $X^{\C}$ is isolated and by Corollary 1 in  \cite{Bialynicki-Birula1}, $X^{\C}$ is  connected and so  it is
exactly one point and we denote it by $p_0$.
 By Proposition \ref{Prop2.2}, we have $\Ch_0(X)\cong \Z$. By Bloch-Srinivas' Proposition \ref{Prop3.2}, we have
\begin{equation}\label{eq3.4}
\Delta_X = \Gamma_0 + \Gamma^1 \in \Ch^n(X\times X)\otimes \Q,
\end{equation}
 where $\supp (\Gamma_0)=X \times p_0$, $p_0$ is the fixed point and $\Gamma^1\subset D\times X$. Since $X$ admits $\C$-action, denoted by $\C\times X\to X, (t,x)\mapsto\phi_t(x)$,
$X\times X$ carries the induced action $\C\times X\times X\to X\times X$ by $(t, (x,y)) \mapsto (\phi_t(x),\phi_t(y))$. Note that both the diagonal $\Delta_X\subset X\times X$  and $ X\times p_0$ carry the induced
$\C$-action by restriction. By applying $\phi_t$ on Equation \eqref{eq3.4}, one gets  $\Gamma^1$ is $\C$-invariant and carries the induced $\C$-action.

To see this, one note that the flow $\phi_t$ induced from the $\C$-action  generates a finite volume complex graph $\mathcal{T}$ in the sense of Harvey and Lawson (see \cite[\S9]{Harvey-Lawson}), where
$$
\mathcal{T}:=\{(t,\phi_t(x),x)\in \C\times X\times X|t\in \C  ~{\rm and}~ x\in X \}\subset \P^1\times X
\times X
$$
 and its closure $\overline{\mathcal{T}}\subset  \P^1\times X \times X$ is a projective variety.   This $\overline{\mathcal{T}}$
 gives a rational equivalence of $\Delta_X-\Gamma_0$ to an algebraic cycle  supported  in $D\times X$ since the limit points of non-fixed
  points supported in proper subvarieties, where $D$ is a divisor of $X$. The algebraic cycle we denote by $\Gamma^1$. We have the similar method to choose $\Gamma^i$ for $i>1$ below.

This means $\Gamma^1=\phi_t(\Gamma^1)\subset\phi_t(D)$ for all $t\in\C$ and so $\Gamma^1\subset \cap_{t\in \C}\phi_t(D)$.
Therefore there are two possible cases: Either $D$ is $\C$-invariant divisor or $D'=\cap_{t\in \C}\phi_t(D)$ is a $\C$-invariant subscheme of lower dimension. In the first case, we can write $D=\sum a_iZ_i$,
where each $Z_i$ is an irreducible variety of codimension one, while in the second case $D'=\sum a_iZ_i'$, where $codim (Z_i')\geq 1$. The invariance of $D$ or $D'$ implies that each $Z_i$ or $Z_i'$  is $\C$-invariant.
 By   taking $p=1,v=n-1$ in Proposition \ref{Prop3..4} and Lemma \ref{lemma3.5}, one gets the surjectivity of
  the map  $\Ch_0(D')_{hom}\otimes\Q\to \Ch_{1}(X)_{hom}\otimes \Q$. If $codim (Z_i')> 1$, then $\Ch^{n-1}(Z_i')=0$.
This implies that  the action of $\Gamma_*$ on $\Ch_1(X)=\Ch^{n-1}(X)$ is zero if the support of the cycle $\Gamma$ is in $ Z_i'\times X$.
 Therefore, we only need to consider such a $\Gamma$  whose support in $Z_i'\times X$ with $codim(Z_i')=1$ and we can assume that $D'=D$ is an invariant divisor, i.e. a scheme  pure of codimension 1 in $X$.
 We write $Z_i'$ as $Z_i$ below.
 Since $Z_i$ admits $\C$-action, we have $\Ch_0(Z_i)\cong \Z$ by Proposition \ref{Prop2.2}.
Again by  Proposition \ref{Prop3..4} and Lemma \ref{lemma3.5},  the map $\Ch_{0}(\sum Z_i)_{hom}\otimes\Q\to \Ch_{1}(X)_{hom}\otimes\Q$ is surjective. Therefore  $\Ch_1(X)_{hom}\otimes \Q=0$.
Since $\Ch_1(X)_{hom}\otimes \Q=0$, there exists a scheme $W_{1}$ of dimension $1$ such that $\Ch_1(W_1)\otimes\Q \twoheadrightarrow \Ch_1(X)\otimes {\Q}$ is surjective.  Note that $W_1$ can be chosen as $\C$-invariant.
Now  $\Ch_0(X)_{hom}\otimes \Q=0$ and $\Ch_1(X)_{hom}\otimes \Q$ imply that the diagonal $\Delta_X$ can be written as
\begin{equation}\label{eq3.6}
\Delta_X = \Gamma_0 + \Gamma_1+ \Gamma^2\in \Ch^n(X\times X)\otimes \Q,
\end{equation}
where $\Gamma^1=\Gamma_1+\Gamma^2$, $\supp(\Gamma_1)\subset V_{n-1}\times W_{1}$ and $\supp(\Gamma^2)\subset V_{n-2}\times X$.  The cycle $\Gamma_1$ is  $\C$-invariant
since it is the restriction of  $\Gamma^1$ on $X\times W_1$, where $\Gamma^1$ is $\C$-invariant. Hence $\Gamma^2=\Gamma^1-\Gamma_1$ is also $\C$-invariant.

As above, $\supp(\Gamma^2)\subset V_{n-2}\times X$ and $\Gamma^2$ is $\C$-invariant. This implies that $V_{n-2}$ is $\C$-invariant. Therefore, $V_{n-2}$ admits a $\C$-action with $V_{n-2}^{\C}=p_0$.
By Proposition \ref{Prop2.2}, $\Ch_0(V_{n-2})_{hom}=0$. By Proposition \ref{Prop3..4} and Lemma \ref{lemma3.5}, $\Ch_{0}(V_{n-2})_{hom}\otimes\Q\to\Ch_2(X)_{hom}\otimes\Q$ is surjective. Hence $\Ch_2(X)_{hom}\otimes\Q=0$.
By induction, we can continue this procedure such that $\Ch_p(X)_{hom}\otimes \Q=0$ for all $p=0,1,...,n=\dim X$.
Therefore, $\Ch_p(X)\otimes \Q\to H_{2p}(X,\Q)$ is injective for all $p$. This completes the proof of the theorem.
\end{proof}

\begin{remark}
Contrary to the nonsingular case, if $X$ is a singular irreducible projective variety admitting a $\C$-action with isolated fixed points, the dimension of $\Ch_1(X)_{hom}\otimes\Q$ as a rational vector space can be infinite and hence $\Ch_1(X)\otimes \Q\to H_{2}(X,\Q)$ is not injective any more.
For example, if $X$ is a cone over a smooth projective variety $Y$ with a nonzero geometric genus.
\end{remark}

\begin{remark}
Comparing to the case that $X$ is a smooth projective variety admitting a $\C$-action, one can get a stronger structure theorem than Theorem \ref{Th3.1} for $X$ if it admits a $\C^*$-action with isolated fixed points (see \cite{Bialynicki-Birula2}).

Note that it is well known that a smooth projective variety with a non-trivial $\C$-action will always admitted a non-trivial action of $\C^*$.
However, it seems hard to find a smooth projective variety with $\C$-action with isolated fixed points without a action of $\C^*$ with isolated fixed point. Most
of the known examples are of $G/P$ type, where $G$ is an algebraic group and $P$ is a parabolic subgroup.
Konarski has found a few examples of such smooth projective varieties in $\P^n$ defined quadratic polynomials (\cite{Konarski2}).
 One of the varieties admit both a $\C$ and  $\C^*$-action  with isolated points, with the same homological group with $\P^5$
 but carries different topological type.
He also pointed out there may exist more examples using higher degree polynomials.

One the other hand, it is a basic open problem in this area whether there exists a smooth variety with a $\C$-action with isolated points but it does not admit any $\C^*$-action isolated fixed points (cf. \cite{Carrell}).
Moreover, a conjecture of Carrell says that \emph{a smooth protective variety that admits a holomorphic vector field with exactly one zero is rational}.
If the Carrell conjecture fails, then any of such counterexamples  is  the example of the above type.

\end{remark}

The next subsection says that Carrell conjecture cannot be disproved through comparing invariants of
a rational oriented cohomological theory (cf. \cite{Panin}). More precisely,
let $X$ be  a $\C$-action with isolated points, the rational coefficients oriented cohomological groups of $X$ coincide with some
smooth rational projective variety $Y$.

\subsection{Applications}

As the first application, we get the following result on  Hodge numbers.
\begin{corollary}[\cite{Carrell-Lieberman}]\label{Cor3.09}
Under the assumption of Theorem  \ref{Thm1}, we have $h^{p,q}(X)=0$ if $p\neq q$.
\end{corollary}
\begin{proof}
This is the classical application of Theorem \ref{Thm1} and  the Bloch-Srinivas diagonal decomposition method and its generalization (see \cite[Th.1.9]{Laterveer}).
\end{proof}

Note that the injectivity of $\Ch_p(X)_{\Q}\to H_{2p}(X,\Q)$ for all $0\leq p\leq \dim X$ implies the surjectivity of $\Ch_p(X)_{\Q}\to H_{2p}(X,\Q)$, since
the Hodge conjecture holds in this case and any element $\alpha\in H_{2p}(X,\Q)\cong H^{2n-2p}(X,\Q)$ is of $(n-p,n-p)$-form by Corollary \ref{Cor3.09}. Hence the cycle class map
is an isomorphism after tensoring the coefficients.
\begin{corollary}\label{Cor3.10}
Under the assumption of Theorem  \ref{Thm1}, we have the isomorphism
 $\Ch_p(X)_{\Q}\cong H_{2p}(X,\Q)$ for all $p\geq 0$.
\end{corollary}

\begin{corollary}\label{Cor3.11}
Under the assumption of Theorem  \ref{Thm1}, we have $$\Phi_{p,k}:L_pH_k(X)_{\Q}\cong H_k(X,\Q)$$ for $k\geq 2p\geq 0$.
\end{corollary}
\begin{proof}
First, we show $\Phi_{p,k}$ is injective for $k\geq 2p\geq 0$. By  Theorem \ref{Thm1}, we have $\Ch_p(X)_{hom}\otimes \Q=0$. This together with the action of diagonal on Lawson homology with rational coefficients, we obtain that the cycle map $\Phi_{p,k}:L_pH_k(X)_{\Q}\to H_k(X,\Q)$ is injective for $k\geq 2p\geq 0$ (see \cite{Peters}).
Now we show $\Phi_{p,k}$ is injective for $k\geq 2p\geq 0$. For $k$ odd, then $H_k(X,\Q)=0$ and so $L_pH_k(X)_{\Q}=0$.
Hence $ L_pH_k(X)_{\Q}\cong H_k(X,\Q)=0$. For $k$ even, set $k=2m$. In this case, we have $2p\leq k=2m$ and so $p\leq m$.
By Corollary \ref{Cor3.10},  we see that  the cycle class map  $\Phi_{m,2m}: L_mH_{2m}(X)_{\Q}\to  H_{2m}(X,\Q)$ is an isomorphism. Moreover,   the map $\Phi_{m,2m}$ factors through $\Phi_{p,2m}:L_pH_{2m}(X)_{\Q}\to  H_{2m}(X,\Q)$
for $0\leq p\leq m$ (see \cite{Friedlander-Mazur}). Therefore, $\Phi_{p,2m}: L_pH_{2m}(X)_{\Q}\to  H_{2m}(X,\Q)$ is surjective and hence $\Phi_{p,k}:L_pH_k(X)_{\Q}\to H_k(X,\Q)$ is surjective  for $k\geq 2p\geq 0$.
This completes the proof of the corollary.
\end{proof}

\begin{remark}
  Under the assumption of Theorem  \ref{Thm1}, we have shown in Proposition \ref{Prop2.2} and Corollary \ref{coro2.11} that
  $\Phi_{0}:\Ch_0(X)\cong H_{0}(X,\Z)\cong \Z$ and $\Phi_{1,k}:L_1H_k(X)\cong H_k(X,\Z)$ for $k\geq 2\geq 0$. However,
  it is still unknown in general whether $\Phi_{p}:\Ch_p(X)\cong H_{2p}(X,\Z)$ for all $p\geq 1$
  and  $\Phi_{p,k}:L_pH_k(X)\cong H_k(X,\Z)$ for $k\geq 2p\geq 4$.
\end{remark}

\begin{remark}
Corollary \ref{Cor3.10} and \ref{Cor3.11} is  a general principle that if $\Ch^*(X\times X)_{hom}\otimes \Q=0$, then all
the rational oriented (co)homology theory of $X$ are isomorphic to the singular homology theory of $X$.
Note the Chow theory and Lawson homology are examples of oriented cohomology theory. Other examples of oriented cohomology theory
are  complex cobordism, complex K-theory, and Morava K-theories(cf. \cite{Panin}).
\end{remark}

Let $X$ be a smooth complex projective variety.
It was shown in \cite[\S 7]{Friedlander-Mazur} that the subspaces
$T_pH_k(X,{\mathbb{Q}})$ form a decreasing filtration (called the \emph{topological filtration}):
$$\cdots\subseteq T_pH_k(X,{\mathbb{Q}})\subseteq T_{p-1}H_k(X,{\mathbb{Q}})
\subseteq\cdots\subseteq
T_0H_k(X,{\mathbb{Q}})=H_k(X,{\mathbb{Q}})$$ and
$T_pH_k(X,{\mathbb{Q}})$ vanishes if $2p>k$.

Denote by
$G_pH_k(X,{\mathbb{Q}})\subseteq H_k(X,{\mathbb{Q}})$ the
$\mathbb{Q}$-vector subspace of $H_k(X,{\mathbb{Q}})$ generated by
the images of mappings $H_k(Y,{\mathbb{Q}})\rightarrow
H_k(X,{\mathbb{Q}})$, induced from all morphisms $Y\rightarrow X$ of
varieties of dimension $\leq k-p$.

The subspaces $G_pH_k(X,{\mathbb{Q}})$ also form a decreasing
filtration (called the \emph{geometric filtration}):
$$\cdots\subseteq G_pH_k(X,{\mathbb{Q}})\subseteq G_{p-1}H_k(X,{\mathbb{Q}})
\subseteq\cdots\subseteq G_0H_k(X,{\mathbb{Q}})\subseteq
H_k(X,{\mathbb{Q}})$$

 Denote by $\tilde{F}_pH_k(X,{\Q})\subseteq
H_k(X,{\Q})$  the maximal sub-(Mixed) Hodge structure of span
$k-2p$. (See \cite{Grothendieck} and \cite{Friedlander-Mazur}.) The sub-${\Q}$ vector spaces
$\tilde{F}_pH_k(X,{\Q})$ form a decreasing filtration of
sub-Hodge structures:
$$\cdots\subseteq \tilde{F}_pH_k(X,{\Q})\subseteq \tilde{F}_{p-1}H_k(X,{\Q})
\subseteq\cdots\subseteq \tilde{F}_0H_k(X,{\Q})\subseteq H_k(X,{\Q})$$ and $\tilde{F}_pH_k(X,{\Q})$ vanishes if $2p>k$. This
filtration is called the \emph{Hodge filtration}.

It was shown by Friedlander and Mazur that
\begin{equation}
T_pH_k(X,{\mathbb{Q}})\subseteq G_pH_k(X,{\mathbb{Q}})\subseteq \tilde{F}_pH_k(X,{\Q})
\end{equation}
holds for any smooth projective variety $X$ and $k\geq 2p\geq0$.

Friedlander and Mazur proposed the following conjecture which relates Lawson homology theory to the central problems in the algebraic cycle theory.
\begin{conjecture}[Friedlander-Mazur \cite{Friedlander-Mazur} ,Grothendieck \cite{Grothendieck}]
For any smooth projective variety $X$ and $k\geq 2p\geq 0$, one has
$$
T_pH_k(X,{\mathbb{Q}})= G_pH_k(X,{\mathbb{Q}})=\tilde{F}_pH_k(X,{\Q}).
$$
\end{conjecture}

Corollary \ref{Cor3.11} implies directly that the Friedlander-Mazur conjecture and the
 Generalized Hodge conjecture (see \cite{Grothendieck}) holds for such a smooth projective variety.
\begin{corollary}
Under the assumption of Theorem  \ref{Thm1}, the Friedlander-Mazur conjecture and the Generalized Hodge conjecture hold for $X$,
i.e., for $k\geq 2p\geq 0$,  we have $$
T_pH_k(X,{\mathbb{Q}})= G_pH_k(X,{\mathbb{Q}})=\tilde{F}_pH_k(X,{\Q}).
$$
\end{corollary}

\emph{Acknowledgements.}
The project was partially sponsored by  STF of Sichuan province, China(2015JQ0007) and NSFC(11771305).

\end{document}